\documentclass[a4paper,10pt]{amsart}

\usepackage[latin1]{inputenc}
\usepackage{color}
\usepackage{url}
\usepackage{pstricks}
\usepackage[bookmarks=false,pdfborder={0 0 0.05}]{hyperref}
\usepackage{multicol}
\usepackage{tikz}
\usetikzlibrary{calc}

\def\Dfn#1{{\sf #1}}

\def\1#1{\operatorname{\bf 1}_k}
\def\pink{\pi_{n,k}}

\def\S#1{\mathcal{S}_{#1}}

\def\SP#1{\mathfrak{S}_{#1}}
\def\RP{\mathcal{RP}}

\def\del{\operatorname{del}}
\def\l{\operatorname{\ell}}
\def\link{\operatorname{link}}
\def\Cat{\operatorname{Cat}}
\def\Dn{\Delta_n}
\def\Dnk{\Delta_{n,k}}

\def\textcross{
	\begin{minipage}{13pt}
		\begin{tikzpicture}[scale=1]
			\cpipedream{0.4}{(0,0)}{0/0/black/black}
		\end{tikzpicture}
	\end{minipage}
}
\def\textelbow{
	\begin{minipage}{13pt}
		\begin{tikzpicture}[scale=1]
			\tpipedream{0.4}{(0,0)}{0/0/black/black}
		\end{tikzpicture}
	\end{minipage}
}

\newcommand{\ngon}[8]{ 

	\foreach \t in {1,...,#1} {
		\coordinate (#7\t) at ($#2+(#8-\t*360/#1:#3)$);
	}

	\foreach \x/\y/\z in {#4}{
		\draw[\z,shorten <=#5pt, shorten >=#5pt] {(#7\x)--(#7\y)};
	}

	\setcounter{intege}{1}
	\pgfmathsetcounter{intege}{1}
	\foreach \object in {#6}{
		\node[inner sep=0pt] at (#7\theintege) {\object};
		\pgfmathsetcounter{intege}{\theintege+1}
		\setcounter{intege}{\theintege}
	}
}

\newcommand{\tikzbox}[8]{ 
  \coordinate (A) at #2;
  \coordinate (B) at ($ (A) + (#1,0) $);
  \coordinate (C) at ($ (A) + (0,#1) $);
  \coordinate (D) at ($ (A) + (#1,#1)$);
  \coordinate (E) at ($ (A) + (#1/2,#1/2) $);

  \draw[fill=#7] (A) rectangle (D);

  \draw[color=#3] (A) -- (B);
  \draw[color=#6] (A) -- (C);
  \draw[color=#4] (B) -- (D);
  \draw[color=#5] (C) -- (D);
  
  \node at (E) []{#8};
}

\newcommand{\blackbox}[3]{ 
  \tikzbox{#1}{#2}{black}{black}{black}{black}{white}{#3}
}

\newcommand{\boxcollection}[3]{ 
  \coordinate (X) at #2;
	
  \foreach \x/\y/\object in {#3} {
		\blackbox{#1}{($ (X) + (#1 * \x, - #1 * \y) $)}{\object};
	}
}

\newcommand{\content}[3]{ 
  \coordinate (C) at #2;
	
  \foreach \x/\y/\object in {#3} {
		\node at ($ (C) + ( #1 * \x, -#1 * \y ) + ( #1 / 2, -#1 / 2 )$) {\object};
	}
}

\newcommand{\tpipedream}[3]{ 
  \coordinate (P) at #2;
	
  \foreach \x/\y/\a/\b in {#3} {
		\coordinate (P1) at ($ (P) + ( #1 * \x , -#1 * \y ) + ( 0      , #1 / 2 ) $);
	  \coordinate (P2) at ($ (P) + ( #1 * \x , -#1 * \y ) + ( #1     , #1 / 2 ) $);
	  \coordinate (P3) at ($ (P) + ( #1 * \x , -#1 * \y ) + ( #1 / 2 , #1     ) $);
	  \coordinate (P4) at ($ (P) + ( #1 * \x , -#1 * \y ) + ( #1 / 2 , 0 ) $);
	  \coordinate (P5) at ($ (P) + ( #1 * \x , -#1 * \y ) + ( #1 / 2 , #1 / 2 ) $);
	  \draw[rounded corners=4, color=\a, thick] (P1) -- (P5) -- (P3);
	  \draw[rounded corners=4, color=\b, thick] (P4) -- (P5) -- (P2);
	}

}

\newcommand{\cpipedream}[3]{ 
  \coordinate (P) at #2;
	
  \foreach \x/\y/\a/\b in {#3} {
		\coordinate (P1) at ($ (P) + ( #1 * \x , -#1 * \y ) + ( 0      , #1 / 2 ) $);
	  \coordinate (P2) at ($ (P) + ( #1 * \x , -#1 * \y ) + ( #1     , #1 / 2 ) $);
	  \coordinate (P3) at ($ (P) + ( #1 * \x , -#1 * \y ) + ( #1 / 2 , #1     ) $);
	  \coordinate (P4) at ($ (P) + ( #1 * \x , -#1 * \y ) + ( #1 / 2 , 0 ) $);
	  \coordinate (P5) at ($ (P) + ( #1 * \x , -#1 * \y ) + ( #1 / 2 , #1 / 2 ) $);
	  \draw[rounded corners=0.2, color=\a, thick] (P1) -- (P5) -- (P3);
	  \draw[rounded corners=0.2, color=\b, thick] (P4) -- (P5) -- (P2);
	}

}

\newcommand{\latticepath}[4]{ 
  \coordinate (L) at #2;
	
  \foreach \x/\y/\a in {#4} {
  	\coordinate (L1) at ($ (L) + ( #1 * \x , #1 * \y ) $);
	  \draw[color=\a, #3] (L) -- (L1);
	  \coordinate (L) at (L1);
	}

}

\definecolor{grey}{rgb}{.7 , .7 , .7}

\newtheorem{theorem}{Theorem}[section]
\newtheorem{statement}[theorem]{Statement}
\newtheorem{corollary}[theorem]{Corollary}

\theoremstyle{definition}
\newtheorem{example}[theorem]{Example}
\newtheorem{remark}[theorem]{Remark}	

\begin{document}
	\title{A new Perspective on $k$-Triangulations}
	\author{Christian Stump}
  \address{LaCIM, Universit\'e du Qu\'ebec \`a Montr\'eal, Montr\'eal (Qu\'ebec), Canada}
  \email{christian.stump@lacim.ca}
	\urladdr{http://homepage.univie.ac.at/christian.stump/}
	\date{\today}
	\thanks{
		The author would like to thank Christian Krattenthaler, Martin Rubey and Luis Serrano for various discussions.
	}
	\begin{abstract}
		We connect $k$-triangulations of a convex $n$-gon to the theory of Schubert polynomials. We use this connection to prove that the simplicial complex with $k$-triangulations as facets is a vertex-decomposable triangulated sphere, and we give a new proof of the determinantal formula for the number of $k$-triangulations.
	\end{abstract}
	\maketitle

	\section{Introduction}
		Let $\Dn$ be the simplicial complex with vertices being diagonals in a convex $n$-gon and facets being \Dfn{triangulations} (i.e., maximal subsets of diagonals which are mutually noncrossing). It is well known that $\Dn$ is a triangulated sphere and moreover that it is the boundary complex of the \emph{dual associahedron}, see e.g. \cite{Lee1989}.
		
		This construction can be generalized using an additional positive integer $k$ with $2k+1 \leq n$. Define a \Dfn{$(k+1)$-crossing} to be a set of $k+1$ diagonals which are mutually crossing. The simplicial complex $\Dnk$ has vertex set
		$$\Gamma_{n,k} := \big\{ \overline{ij} : k < | i-j | < n-k \big\}.$$
		Its facets are given by \Dfn{$k$-triangulations}, i.e., maximal subsets of diagonals in $\Gamma_{n,k}$ which do not contain a $(k+1)$-crossing. The reason for restricting the set of diagonals is simply that all other diagonals cannot be part of a $(k+1)$-crossing and thus the resulting simplicial complex would be a join of $\Dnk$ and an $nk$-simplex. The complex $\Dnk$ has been studied by several authors, see e.g. \cite{DKM2002, J2005, JW2007, Kra2006, Nak2000, Rub2006}; a very interesting survey of what is known about $k$-triangulations can be found in \cite{PS2009}.

		\begin{theorem}[Dress, Koolen, Moulton \cite{DKM2002}, Nakamigawa \cite{Nak2000}]\label{th:dim}
			Every $k$-triangulation contains $k(n-2k-1)$ diagonals. In particular, $\Dnk$ is pure of dimension $k(n-2k-1)-1$.
		\end{theorem}
		\begin{theorem}[Jonsson \cite{J2005}, Krattenthaler \cite{Kra2006}]\label{th:number}
			The number of facets in $\Dnk$ (which is the number of $k$-triangulations of a convex $n$-gon) is given by
			$$\det \begin{pmatrix}
												\Cat_{n-2} & \cdots & \Cat_{n-k-1} \\
			                	\vdots & \ddots & \vdots \\
												\Cat_{n-k-1} & \cdots & \Cat_{n-2k}
			               \end{pmatrix}
											= \prod_{1 \leq i \leq j < n-2k} \frac{i+j+2k}{i+j},
			$$
			where $\Cat_m = \frac{1}{m+1}\binom{2m}{m}$ denotes the $m$-th \Dfn{Catalan number}.
		\end{theorem}
		The first proof of Theorem~\ref{th:number} by J.~Jonsson is rather involved and uses multiple results on moon polyominos whereas the second by C.~Krattenthaler uses growth diagrams and proves that the number of $k$-triangulations equals the number of certain fans of nonintersecting Dyck paths. The number of such fans can be counted using the Lindström-Gessel-Viennot Lemma on families of nonintersecting lattice paths, see for example Sections~2.6.1 and 2.6.2. in \cite{Man2001}. This gives a natural explanation of the determinantal formula.
		
		For a simplicial complex $\Delta$ and a face $F \in \Delta$, define
		\begin{itemize}
			\item the \Dfn{deletion} of $F$ from $\Delta$ by $\del(\Delta,F) := \{ G \in \Delta : G \cap F = \emptyset \}$,
			\item the \Dfn{link} of $F$ in $\Delta$ by $\link(\Delta,F) := \{ G \in \Delta : G \cap F = \emptyset, G \cup F \in \Delta \}$.
		\end{itemize}
		Moreover, $\Delta$ is called \Dfn{vertex-decomposable} if $\Delta$ is pure and either (1) $\Delta = \{ \emptyset \}$ or (2) $\del(\Delta,v)$ and $\link(\Delta,v)$ are both vertex-decomposable for some vertex $v \in \Delta$. Vertex-decomposability was introduced by L.J.~Billera and J.S.~Provan in \cite{BP1979}, where they moreover showed that vertex-decomposability implies shellability.

		The following fact about the simplicial complex $\Dnk$ was proved by J.~Jonsson in an unpublished manuscript communicated in \cite{J2003}. He informed me that there exists as well an unpublished manuscript by A.~Dress, S.~Grünewald and V.~Moulton containing a proof.
		\begin{statement}[Jonsson]\label{th:vertex-decomposable}
			$\Dnk$ is a vertex-decomposable triangulated sphere.
		\end{statement}

		In this paper we want to connect $k$-triangulations in a surprisingly simple way to the theory of Schubert polynomials and thereby present a proof of Statement~\ref{th:vertex-decomposable} and as well another viewpoint on C.~Krattenthaler's proof of Theorem~\ref{th:number}.

		To this end define, for any permutation $\pi \in \S{n}$ and any word $Q$ in the simple generators $s_i := (i,i+1) \in \S{n}$, a simplicial complex which was introduced by A.~Knutson and E.~Miller in the context of Schubert polynomials in \cite[Definition 1.8.1]{KM2005} and further studied in \cite{KM2004}. The \Dfn{subword complex} $\Delta(Q, \pi)$ is defined to be the simplicial complex with vertices $v_i$ being labelled by $w_i$ for every letter $w_i \in Q$. Note that $v_i \neq v_j$ for $i \neq j$, even if the letters $w_i = w_j$ coincide. A subword of $Q$ forms a facet of $\Delta(Q,\pi)$ if and only if its complement is a reduced word for $\pi$.
		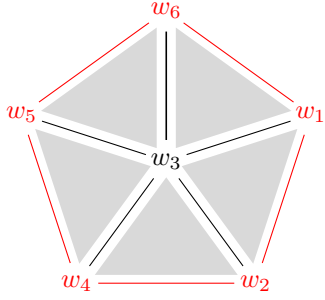
\begin{figure}
			\centering
			\begin{tikzpicture}[scale=1]
				\coordinate (allo) at (0,0);
				\newcounter{intege}
				\ngon{5}{(allo)}{2}
						{1/2/red,2/3/red,3/4/red,4/5/red,5/1/red}
					{8}{$\color{red} w_1$,$\color{red} w_2$,$\color{red} w_4$,$\color{red} w_5$,$\color{red} w_6$}{first}{90}
				\node[inner sep=0pt] at (allo) {$\color{black} w_3$};
				\draw[black,shorten <=8pt, shorten >=8pt] {(first1)--(allo)};
				\draw[black,shorten <=8pt, shorten >=8pt] {(first2)--(allo)};
				\draw[black,shorten <=8pt, shorten >=8pt] {(first3)--(allo)};
				\draw[black,shorten <=8pt, shorten >=8pt] {(first4)--(allo)};
				\draw[black,semithick,shorten <=8pt, shorten >=8pt] {(first5)--(allo)};

				\coordinate (x1) at ($0.33*(first1) + 0.33*(first2)$);
				\coordinate (y1) at ($0.2*(x1) + 0.8*(first1)$);
				\coordinate (y2) at ($0.2*(x1) + 0.8*(first2)$);
				\coordinate (y3) at ($0.2*(x1)$);
				\fill[fill opacity=0.15] (y1) -- (y2) -- (y3);

				\coordinate (x1) at ($0.33*(first2) + 0.33*(first3)$);
				\coordinate (y1) at ($0.2*(x1) + 0.8*(first2)$);
				\coordinate (y2) at ($0.2*(x1) + 0.8*(first3)$);
				\coordinate (y3) at ($0.2*(x1)$);
				\fill[fill opacity=0.15] (y1) -- (y2) -- (y3);

				\coordinate (x1) at ($0.33*(first3) + 0.33*(first4)$);
				\coordinate (y1) at ($0.2*(x1) + 0.8*(first3)$);
				\coordinate (y2) at ($0.2*(x1) + 0.8*(first4)$);
				\coordinate (y3) at ($0.2*(x1)$);
				\fill[fill opacity=0.15] (y1) -- (y2) -- (y3);

				\coordinate (x1) at ($0.33*(first4) + 0.33*(first5)$);
				\coordinate (y1) at ($0.2*(x1) + 0.8*(first4)$);
				\coordinate (y2) at ($0.2*(x1) + 0.8*(first5)$);
				\coordinate (y3) at ($0.2*(x1)$);
				\fill[fill opacity=0.15] (y1) -- (y2) -- (y3);

				\coordinate (x1) at ($0.33*(first5) + 0.33*(first1)$);
				\coordinate (y1) at ($0.2*(x1) + 0.8*(first5)$);
				\coordinate (y2) at ($0.2*(x1) + 0.8*(first1)$);
				\coordinate (y3) at ($0.2*(x1)$);
				\fill[fill opacity=0.15] (y1) -- (y2) -- (y3);

			\end{tikzpicture}
			\caption{$\Delta(Q,\pi)$ as the join of $\Delta(Q',\pi)$ with the centered point $w_3$.}
			\label{fig:sachertorte}
		\end{figure}
		\begin{example}[following Example 1.8.2 in \cite{KM2005}]\label{ex:pentagon}
			Let
			$$Q' = w_1 w_2 w_4 w_5 w_6 := s_3 s_2 s_3 s_2 s_3$$
			and let $\pi=[1,4,3,2]$. Consider $\Delta( Q', \pi )$: as $\pi$ has the two reduced expressions $s_2 s_3 s_2 = s_3 s_2 s_3$, the reduced words in $Q'$ for $\pi$ are
			$$ w_1 w_2 w_4, w_2 w_4 w_5, w_4 w_5 w_6, w_1 w_5 w_6, \text{ and } w_1 w_2 w_6.$$
			Therefore, $\Delta( Q', \pi )$ is a pentagon as shown in {\color{red}red} in Figure~\ref{fig:sachertorte}. If we instead consider
			$$Q = w_1 w_2 w_3 w_4 w_5 w_6 = s_3 s_2 s_1 s_3 s_2 s_3,$$
			we obviously obtain $\Delta(Q,\pi)$ as shown in Figure~\ref{fig:sachertorte} to be a $5$-piece cake (which is the join of $\Delta(Q',\pi)$ with the centered point $w_3 = s_1)$.
		\end{example}
		A.~Knutson and E.~Miller proved the following two beautiful theorems concerning subword complexes in \cite[Theorem~2.5, Corollary~3.8]{KM2004}.
		\begin{theorem}[Knutson, Miller]\label{th:shellable}
			Any subword complex $\Delta(Q,\pi)$ is pure of dimension $\l(Q) - \l(\pi) - 1$. It is moreover vertex-decomposable, and thus shellable.
		\end{theorem}
		In the theorem, $\l(Q)$ is simply the number of letters in $Q$ and $\l(\pi)$ is the Coxeter length of $\pi$ which is the length of any minimal expression for $\pi$. The pureness and the dimension of $\Delta(Q,\pi)$ then follow immediately, the fact that it is vertex-decomposable is proved by showing that both the link and the deletion of the first letter in $Q$ are again subword complexes and thus vertex-decomposable by induction.

		For the second theorem, we need the notion of the Demazure product. Let $\pi \in \S{n}$, $Q$ be a (not necessarily reduced) word for $\pi$, and $Q'$ be obtained from $Q$ by adding an additional generator $s$ at the end. The \Dfn{Demazure product} $\delta$ of the empty word is defined to be the identity permutation, and
		$$ \delta(Q') =
		\begin{cases}
			\delta(Q) s &\text{ if } \ell(\pi s) > \ell(\pi), \text{ and}\\
			\delta(Q) &\text{ if } \ell(\pi s) < \ell(\pi).\\
		\end{cases}
		$$
		\begin{theorem}[Knutson, Miller]\label{th:sphere}
			The complex $\Delta(Q,\pi)$ is a triangulated sphere if $\delta(Q) = \pi$, and it is a triangulated ball otherwise.
		\end{theorem}

	\section{Results}
		The main theorem of this paper is the following:
		\begin{theorem}\label{th:main}
			Let
			$$Q_{n,k} := s_{n-k-1} \cdots s_1 \hspace{5pt} s_{n-k-1} \cdots s_2 \hspace{5pt} \cdots \hspace{5pt} s_{n-k-1} s_{n-k-2} \hspace{5pt} s_{n-k-1}$$
			be the \Dfn{triangulated reduced word} for $\pi_{n-k,0} := [n-k,\ldots,1] \in \S{n-k}$ and let
			$$\pink := [1,\ldots,k,n-k,n-k-1,\ldots,k+1] \in \S{n-k}.$$
			Then
			$$\Dnk = \Delta( Q'_{n,k}, \pink ),$$
			where $Q'_{n,k}$ is obtained from $Q_{n,k}$ by deleting all letters $s_i$ for $1 \leq i \leq k$.
		\end{theorem}
		In Example~\ref{ex:pentagon}, the case of $n=5$ and $k=1$ is discussed with $Q = Q_{5,1}$ and $Q' = Q'_{5,1}$.
		\begin{corollary}\label{cor:1}
			Statement~\ref{th:vertex-decomposable} holds.
		\end{corollary}
		\begin{corollary}\label{cor:2}
			The number of facets in $\Dnk$ is given by the determinantal expression in Theorem~\ref{th:number}.
		\end{corollary}
	\section{Proofs}
		In this section we will provide some further background, proofs and several remarks.
		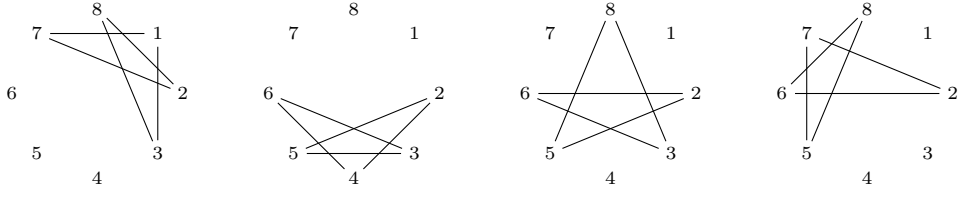
\begin{figure}
			\centering
			\begin{tikzpicture}[scale=0.75]
				\coordinate (c1) at (0,0);
				\coordinate (c2) at (4.5,0);
				\coordinate (c3) at (9,0);
				\coordinate (c4) at (13.5,0);
				\ngon{8}{(c1)}{1.5}
						{3/8/black,8/2/black,2/7/black,7/1/black,1/3/black}
					{5}{{\scriptsize 1},{\scriptsize 2},{\scriptsize 3},{\scriptsize 4},{\scriptsize 5},{\scriptsize 6},{\scriptsize 7},{\scriptsize 8}}{first}{90}
				\ngon{8}{(c2)}{1.5}
						{4/6/black,6/3/black,3/5/black,5/2/black,2/4/black}
					{5}{{\scriptsize 1},{\scriptsize 2},{\scriptsize 3},{\scriptsize 4},{\scriptsize 5},{\scriptsize 6},{\scriptsize 7},{\scriptsize 8}}{first}{90}
				\ngon{8}{(c3)}{1.5}
						{5/8/black,8/3/black,3/6/black,6/2/black,2/5/black}
					{5}{{\scriptsize 1},{\scriptsize 2},{\scriptsize 3},{\scriptsize 4},{\scriptsize 5},{\scriptsize 6},{\scriptsize 7},{\scriptsize 8}}{first}{90}
				\ngon{8}{(c4)}{1.5}
						{6/8/black,8/5/black,5/7/black,7/2/black,2/6/black}
					{5}{{\scriptsize 1},{\scriptsize 2},{\scriptsize 3},{\scriptsize 4},{\scriptsize 5},{\scriptsize 6},{\scriptsize 7},{\scriptsize 8}}{first}{90}
			\end{tikzpicture}
			\caption{Four $2$-stars with $5$ vertices in the $8$-gon.}
			\label{fig:kstar}
		\end{figure}
		One key ingredient is a property of $k$-triangulations discovered by V.~Pilaud and F.~Santos in \cite[Theorem 1.4]{PS2009}. A $k$-star is a polygon on $2k+1$ clockwise ordered different vertices $v_1,\ldots,v_{2k+1}$ of a convex $n$-gon (with the condition that $2k+1 \leq n$) where $v_i$ and $v_j$ are connected if and only if $| i - j | \in \{ k, k+1 \}$, see Figure~\ref{fig:kstar} for four examples.
		\begin{theorem}[Pilaud, Santos] \label{th:stars}
			A $k$-triangulation $T$ together with its boundary $\big\{ \overline{ij} : |i-j| \in \{k,n-k\}\big\}$ contains exactly $n-2k$ $k$-stars. Moreover, every diagonal in $T$ belongs to exactly two $k$-stars, and every diagonal in the boundary belongs to a unique $k$-star.
		\end{theorem}
		\begin{figure}
			\centering
			$\begin{array}{ccc}
			  \begin{tikzpicture}[scale=1]
				\coordinate (allo) at (0,0);
				\ngon{8}{(allo)}{2}
						{
						1/3/grey,2/4/grey,3/5/grey,4/6/grey,5/7/grey,6/8/grey,7/1/grey,8/2/grey,
						2/5/black,2/6/black,2/7/black,3/6/black,3/8/black,5/8/black}
					{6}{1,2,3,4,5,6,7,8}{first}{90}
			\end{tikzpicture}
			&\qquad&
			\begin{tikzpicture}[scale=1]
				\boxcollection{0.5}{(0,0)}{
					0/0/$\circ$,1/0/$\circ$,2/0/$+$,3/0/,4/0/$+$,5/0/$\circ$,6/0/$\circ$,
					0/1/$\circ$,1/1/$+$,2/1/,3/1/,4/1/$\circ$,5/1/$\circ$,
					0/2/,1/2/$+$,2/2/$+$,3/2/$\circ$,4/2/$\circ$,
					0/3/,1/3/$+$,2/3/$\circ$,3/3/$\circ$,
					0/4/,1/4/$\circ$,2/4/$\circ$,
					0/5/$\circ$,1/5/$\circ$,
					0/6/$\circ$}
				\content{0.5}{(0,1)}{
					0/0/$1$,1/0/$2$,2/0/$3$,3/0/$4$,4/0/$5$,5/0/$6$,6/0/$7$,
					-1/1/$8$,-1/2/$7$,-1/3/$6$,-1/4/$5$,-1/5/$4$,-1/6/$3$,-1/7/$2$}
				\latticepath{0.5}{(0,-3)}{thick}{
					1/0/black,0/1/black,1/0/black,0/1/black,1/0/black,0/1/black,1/0/black,0/1/black,1/0/black,0/1/black,1/0/black,0/1/black,1/0/black,0/1/black,-7/0/black,0/-7/black}
				\content{0.5}{(0,-3)}{
					0/0/}
			\end{tikzpicture}
			\\
			(a) & & (b) \\
			\\
			\begin{tikzpicture}[scale=1]
				\tpipedream{0.5}{(0,0)}{
					0/0/red/red,1/0/red/green,2/0/green/green,4/0/green/green,5/0/green/red,6/0/red/white,
					0/1/red/green,1/1/green/green,4/1/green/red,5/1/red/white,
					1/2/green/green,2/2/green/green,3/2/green/red,4/2/red/white,
					1/3/green/green,2/3/green/red,3/3/red/white,
					1/4/green/red,2/4/red/white,
					0/5/green/red,1/5/red/white,
					0/6/red/white}
				\cpipedream{0.5}{(0,0)}{
					3/0/green/green,
					2/1/green/green,3/1/green/green,
					0/2/green/green,
					0/3/green/green,
					0/4/green/green}
				\content{0.5}{(0,1)}{
					0/0/$1$,1/0/$2$,2/0/$3$,3/0/$4$,4/0/$5$,5/0/$6$,6/0/$7$,
					-1/1/$8$,-1/2/$7$,-1/3/$6$,-1/4/$5$,-1/5/$4$,-1/6/$3$,-1/7/$2$}
			\end{tikzpicture}
			&\qquad&
			\begin{tikzpicture}[scale=1]
				\tpipedream{0.5}{(0,0)}{
					0/0/black/black,1/0/black/black,2/0/black/black,4/0/black/black,5/0/black/white,
					0/1/black/black,1/1/black/black,4/1/black/white,
					1/2/black/black,2/2/black/black,3/2/black/white,
					1/3/black/black,2/3/black/white,
					1/4/black/white,
					0/5/black/white}
				\cpipedream{0.5}{(0,0)}{
					3/0/black/black,
					2/1/black/black,3/1/black/black,
					0/2/black/black,
					0/3/black/black,
					0/4/black/black}
				\content{0.5}{(0,1)}{
					0/0/$1$,1/0/$2$,2/0/$3$,3/0/$4$,4/0/$5$,5/0/$6$,
					-1/1/$1$,-1/2/$2$,-1/3/$6$,-1/4/$5$,-1/5/$4$,-1/6/$3$}
				\content{0.5}{(0,-2.6)}{
					0/0/}
			\end{tikzpicture}
			\\
				(c) & & (d)
			\end{array}$
			\caption{A $2$-triangulation, its interpretation in terms of the filling of a staircase diagram, the translation into a pipe diagram and the associated reduced pipe dream for $\pi = [1,2,6,5,4,3]$.}
			\label{fig:pipe dream}
		\end{figure}
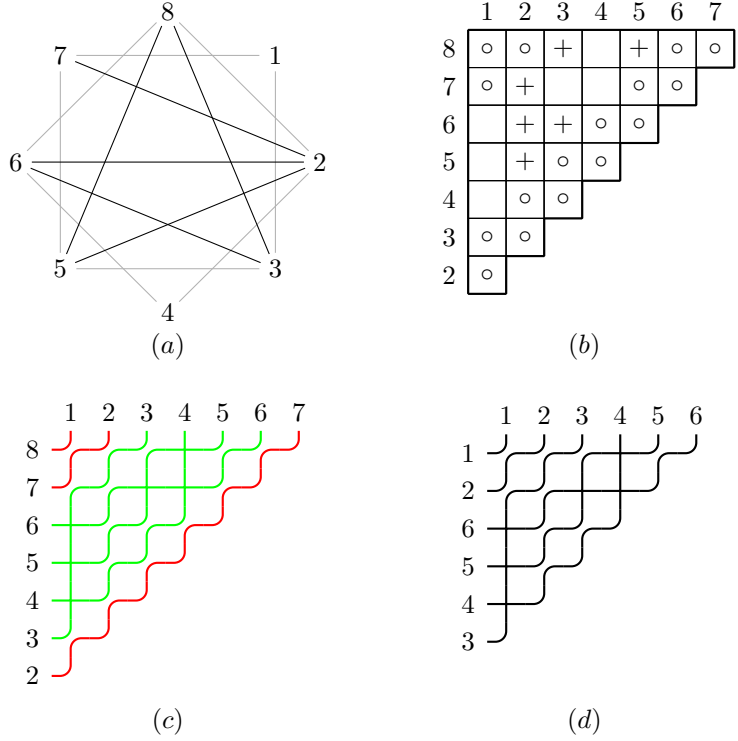
		The next step is to interpret the theorem in terms of fillings of a staircase diagram which can be seen as the Ferrers diagram of the partition $(n-1,n,\ldots,2,1)$. It is well known that a $k$-triangulation $T$ of the $n$-gon can be encoded as a $(+,\circ,\hspace{10pt})$-filling of the staircase diagram, where a box $(i,j)$ is marked with a
		\begin{itemize}
			\item $\circ$ if $\overline{ij}$ is not a diagonal in $\Gamma_{n,k}$,
			\item	$+$ if $\overline{ij}$ is a diagonal in $T$,
			\item and is left blank otherwise.
		\end{itemize}
		Figures~\ref{fig:pipe dream}(a) and (b) show an example, which is the same $2$-triangulation of the $8$-gon as in \cite[Figure~19]{PS2009}. The diagonals in the $2$-triangulation are drawn in black and the $8$ boundary diagonals are added in grey. The four $2$-stars in the $2$-triangulation are exactly the stars shown in Figure~\ref{fig:kstar} given by
		\begin{align*}
			3-8-2-7-1-3, &\quad 4-6-3-5-2-4,\\
			5-8-3-6-2-5, &\quad 6-8-5-7-2-6.
		\end{align*}
		To see the stars in this filling, we replace every $\circ$ and every $+$ by two turning pipes $\textelbow$ and every empty box by two crossing pipes $\textcross$ as shown in Figure~\ref{fig:pipe dream}(c). Now, a resulting pipe starting at the top of column $i$ and ending at the left of row $\pi(i)$ is
		\begin{itemize}
		 \item an \Dfn{outer pipe} if $i \leq k$ or if $i > n-k$ and thus $\pi(i) = n+1-i$, or
		 \item an \Dfn{inner pipe} if $k < i \leq n-k$ and thus $\pi(i) = i$.
		\end{itemize}
		Here, $\pi(i) = n+1-i$ for outer pipes is implied by the fact that they completely consist of turning pipes. Moreover, inner pipes represent exactly the stars in Theorem~\ref{th:stars} and thus consist of $2k+1$ turning and of $n-2k-1$ crossing pieces and connect $i$ on the top with $\pi(i) = i$ on the left. In Figure~\ref{fig:pipe dream}(c), outer pipes are drawn in red and inner pipes in green. The later represent the four $2$-stars in the $2$-triangulation starting in columns $3$ through $6$. Observe moreover, that boxes corresponding to diagonals in $\Gamma_{n,k}$ contain two inner pipes (crossing or turning), while boxes corresponding to the boundary contain one inner and one outer turning pipe, and all other boxes contain two outer turning pipes.

		We have almost reached reduced pipe dreams (or $rc$-graphs) as defined for example in \cite[Section 1.4]{KM2005}. A \Dfn{pipe dream} of size $n$ is a filling of the boxes in the staircase diagram $(n-1,\ldots,1)$ with two crossing pipes $\textcross$ or with two turning pipes $\textelbow$, see Figure~\ref{fig:pipe dream}(d) for an example. The permutation $\pi(D)$ of a pipe dream $D$ is obtained by writing the integers $1$ up to $n$ which appear on top again at the end of the associated pipe and read the permutation from top left to bottom left. For example, the permutation for the pipe dream in Figure~\ref{fig:pipe dream}(d) is $[1,2,6,5,4,3]$. A pipe dream is \Dfn{reduced} if two pipes cross at most once. For a given permutation $\pi$, denote the set of all reduced pipe dreams for $\pi$ by $\RP(\pi)$. Reduced pipe dreams play a central role in the combinatorics of \Dfn{Schubert polynomials}, which can be defined as
		\begin{align}
			\SP{\pi}(x_1,\ldots,x_n) = \sum_{ D \in \RP(\pi) } x^D, \label{eq:Schubert}
		\end{align}
		where $x^D := \prod_{(i,j) \in D} x_i$ and where we say $(i,j) \in D$ if $(i,j)$ is filled in $D$ with a $\textcross$. Further background on Schubert polynomials and their combinatorics can be found e.g. in \cite{LS1985,BB1993,FK1997,Man2001}.

		\begin{theorem}\label{th:bijection}
			$k$-triangulations of the $n$-gon are in canonical bijection with reduced pipe dreams for $\pink$.
		\end{theorem}
		\begin{proof}
			We have already constructed a pipe dream $D$ from a $k$-triangulation in a canonical way. We have seen that the outer pipes force $\pi(D)$ to fix $i$ for $i \notin \{k+1,\ldots,n-k\}$ as we read the permutation in one-line notation from top to bottom. For the inner pipes, we obtain that $\pi(D)$ maps $i$ to $n+1-i$ for $i \in \{k+1,\ldots,n-k\}$. To see that $D$ is reduced, observe that every inner pipe contains $2k+1$ turning and $n-2k-1$ crossing pipes. As there are $n-2k$ inner pipes and any two must cross, this implies that no two pipes can cross twice. Lastly, we have deleted the outer $k-1$ outer pipes in the diagram as they do not contribute to the reduced pipe dream, compare Figures~\ref{fig:pipe dream}(c) and (d).

			It is left to show that all pipe dreams for $\pink$ arise that way. Theorem~\ref{th:dim} implies that every collection of $k(n-2k-1)$ diagonals in $\Gamma_{n,k}$ not containing a $(k+1)$-crossing is a $k$-triangulation. A $(k+1)$-crossing in the $n$-gon translates to a strict north-east chain of $(k+1)$ $+$'s in the staircase diagram which is completely contained in a maximal rectangle inside the staircase diagram, see \cite{J2005}. E.g., the marked boxes $(3,6)$ and $(5,8)$ in Figure~\ref{fig:pipe dream}(b) form a strict north-east chain of length $2$. But adding the box $(2,5)$ does not turn it into a strict north-east chain of length $3$, as the box $(5,5)$ is not contained in the staircase diagram. The corresponding diagonals in Figure~\ref{fig:pipe dream}(a) have the same behavior with respect to crossings.
			
			Thus, any filling with $\circ$'s in boxes $(i,j)$ for $\overline{ij} \notin \Gamma_{n,k}$ and $+$'s in $k(n-2k-1)$ boxes, which does furthermore not contain such a strict north-east chain of length $k+1$ comes from a $k$-triangulation. It follows from \cite[Theorem~B]{KM2005} that the collection of turning pipes in a reduced pipe dream for $\pink$ does not contain such a strict north-east chain of length $k+1$. This can be found as well in \cite[Theorem~3]{JM2008}, where it is shown that the collection of crossing pipes in a pipe dream for $\pi_{n,k}$ intersects every strict north-east chain (called \emph{antidiagonal} in \cite{JM2008}) of length $k+1$. Thus, the collection of turning pipes cannot contain a strict north-east chain of length $k+1$.

			As a reduced pipe dream for $\pi_{n,k}$ contains $\binom{n-2k}{2}$ boxes with crossing pipes and thus $k(n-2k-1)$ boxes with turning pipes inside $\Gamma_{n,k}$, the statement follows.
		\end{proof}

		We are now in the position to prove the main theorem of this paper.
		\begin{proof}[Proof of Theorem~\ref{th:main}]
			To every box in the staircase diagram, one can associate a simple generator by filling the diagram with the triangulated reduced word $Q_{n,k}$ row by row from top to bottom and from right to left. One obtains a reduced word for some $\pi$ by reading the simple generators in boxes filled with $\textcross$'s in a reduced pipe dream for $\pi$. In other words, the pipe dream can be seen as a \emph{braid} for $\pi$. For example, the reduced word obtained from the pipe dream in Figure~\ref{fig:pipe dream}(d) reads $s_4 | s_5 s_4 | s_3  | s_4 | s_5$, where a separator $|$ indicates that we have started reading the next row from right to left. We have thus a one-to-one correspondence between subwords of $Q_{n,k}$ which are reduced words for $\pi$ and reduced pipe dreams for $\pi$. By the canonical bijection in Theorem~\ref{th:bijection} and the definition of subword complexes, we can delete all letters in $Q$ not contained in $\Gamma_{n,k}$ and obtain that $\Delta_{n,k}$ is equal to $\Delta(Q_{n,k}', \pink)$, as desired.
		\end{proof}		
		\begin{proof}[Proof of Corollary~\ref{cor:1}]
			The fact that $\Dnk$ is vertex-decomposable follows immediately from the main theorem together with Theorem~\ref{th:shellable}. To prove that $\Dnk$ is moreover a triangulated sphere, we use that $\pink$ is the longest element in the standard parabolic subgroup generated by the generators in $Q'_{n,k}$, and therefore $\delta(Q_{n,k}') \leq \pi_ {n,k}$ in Bruhat order. As $Q'_{n,k}$ contains a reduced expression for $\pink$ as a suffix, \cite[Lemma~3.4(i)]{KM2004} tells us that $\delta(Q'_{n,k}) \geq \pink$. By Theorem~\ref{th:sphere}, this completes the proof.
		\end{proof}
		\begin{corollary}
			The number of facets in $\Dnk$ is given by the Schubert polynomial $\SP{\pink}$ evaluated at $1$. Moreover, Corollary~\ref{cor:2} holds.
		\end{corollary}
		\begin{proof}
			The fact that $k$-triangulations are counted by $\SP{\pi}(1,\ldots,1)$ follows from Theorem~\ref{th:bijection}, together with the combinatorial description of Schubert polynomials as shown in \eqref{eq:Schubert}. The counting formula then follows with \cite[Lemma~1.2]{FK1997}.
		\end{proof}
		\begin{remark}
			A key step in the proof that $\SP{\pi}(1,\ldots,1)$ is counted by the determinantal formula is that $\pi$ is vexillary and thus, $\SP{\pi}$ equals a \emph{flagged Schur function}, see \cite[Corollary~2.6.10]{Man2001}. This proof is related to C.~Krattenthaler's proof of the determinantal formula in \cite{Kra2006} in the sense that it is as well \emph{not} bijective (as it involves a non-bijective induction) and that it uses the Lindström-Gessel-Viennot Lemma on families of nonintersecting lattice paths. Using a more direct approach, one can obtain a purely bijective construction. This will be treated in a joint paper with L.~Serrano. We will as well show how most of the properties of $k$-triangulations can be reobtained only using the combinatorics of Schubert polynomials.
		\end{remark}
	\bibliographystyle{amsalpha} 
	\bibliography{../mrabbrev,../bibliography}

\providecommand{\bysame}{\leavevmode\hbox to3em{\hrulefill}\thinspace}
\providecommand{\MR}{\relax\ifhmode\unskip\space\fi MR }
\providecommand{\MRhref}[2]{%
  \href{http://www.ams.org/mathscinet-getitem?mr=#1}{#2}
}
\providecommand{\href}[2]{#2}
\begin{thebibliography}{DKM02}

\bibitem[BB93]{BB1993}
N.~Bergeron and S.~Billey, \emph{R{C}-graphs and {S}chubert polynomials},
  Experiment. Math. \textbf{2} (1993), no.~4, 257--269.

\bibitem[BP79]{BP1979}
L.J. Billera and J.S. Provan, \emph{A decomposition property for simplicial
  complexes and its relation to diameters and shellings}, Second International
  Conference on Combinatorial Mathematics, New York Acad. Sci. (1979), 82--85.

\bibitem[DKM02]{DKM2002}
A.~Dress, J.~Koolen, and V.~Moulton, \emph{On line arrangements in the
  hyperbolic plane}, Europ. J. Comb. \textbf{23} (2002), 549--557.

\bibitem[FK97]{FK1997}
S.~Fomin and A.~Kirillov, \emph{Reduced words and plane partitions}, J.
  Algebraic Combin. \textbf{6} (1997), no.~4, 311--319.

\bibitem[JM08]{JM2008}
N.~Jia and E.~Miller, \emph{Duality of antidiagonals and pipe dreams}, S{\'e}m.
  Lothar. Combin. \textbf{58} (2008), B58e.

\bibitem[Jon03]{J2003}
J.~Jonsson, \emph{Generalized triangulations of the n-gon}, Unpublished
  manuscript, abstract in Mathematisches Forschungsinstitut Oberwolfach, Report
  No. 16/2003 (2003).

\bibitem[Jon05]{J2005}
\bysame, \emph{Generalized triangulations and diagonal-free subsets of stack
  polyominos}, J. Comb. Theory, Ser. A \textbf{112} (2005), 117--142.

\bibitem[JW07]{JW2007}
J.~Jonsson and V.~Welker, \emph{A spherical initial ideal for pfaffians},
  Illinois J. Math. \textbf{51} (2007), no.~4, 1397--1407.

\bibitem[KM04]{KM2004}
A.~Knutson and E.~Miller, \emph{Subword complexes in {C}oxeter groups}, Adv.
  Math. \textbf{184} (2004), no.~1, 161--176.

\bibitem[KM05]{KM2005}
\bysame, \emph{Gr{\"o}bner geometry of {S}chubert polynomials}, Ann. of Math.
  \textbf{161} (2005), no.~3, 1245--1318.

\bibitem[Kra06]{Kra2006}
C.~Krattenthaler, \emph{Growth diagrams, and increasing and decreasing chains
  in fillings of {F}errers shapes}, Adv. in Appl. Math. \textbf{37} (2006),
  404--431.

\bibitem[Lee89]{Lee1989}
C.~Lee, \emph{The associahedron and triangulations of the n-gon}, European J.
  Combin. \textbf{10} (1989).

\bibitem[LS85]{LS1985}
A.~Lascoux and M.-P. Sch{\"u}tzenberger, \emph{{S}chubert polynomials and the
  {L}ittlewood-{R}ichardson rule}, Lett. Math. Phys. \textbf{10} (1985),
  no.~2--3, 111--124.

\bibitem[Man01]{Man2001}
L.~Manivel, \emph{Symmetric functions, {S}chubert polynomials and degeneracy
  loci}, SMF/AMS Texts and Monographs \textbf{6} (2001).

\bibitem[Nak00]{Nak2000}
T.~Nakamigawa, \emph{A generalization of diagonal flips in a convex polygon},
  Theor. Comp. Sci. \textbf{235} (2000), 271--282.

\bibitem[PS09]{PS2009}
V.~Pilaud and F.~Santos, \emph{Multitriangulations as complexes of star
  polygons}, Discrete Comput. Geom. \textbf{41} (2009), no.~2, 284--317.

\bibitem[Rub06]{Rub2006}
M.~Rubey, \emph{Increasing and decreasing sequences in fillings of moon
  polyominoes}, to appear in Adv. in Appl. Math., available at {\tt
  arXiv:math/0604140} (2006).

\end{thebibliography}
\end{document}